\documentclass[12pt]{amsart}
\usepackage[margin=1.25in]{geometry}

\usepackage{amsmath}
\usepackage{amsfonts}
\usepackage{amssymb}
\usepackage{amscd}
\usepackage{graphicx}
\usepackage[abbrev,alphabetic]{amsrefs}
\RequirePackage[dvipsnames,usenames]{color}
\usepackage{soul,xcolor}
\setstcolor{red}
\usepackage{stmaryrd}
\usepackage{multicol}

\RequirePackage{hyperref}

\newcommand\hidden[1]{} 
\usepackage{mathtools}
\newtagform{hidden}[\hidden]{}{}

\usepackage{hyperref}

\usepackage{amsthm}
\usepackage{comment}
\usepackage[all,cmtip]{xy}
\usepackage{tikz-cd}
\usetikzlibrary{cd}

\makeatletter
\def\@tocline#1#2#3#4#5#6#7{\relax
  \ifnum #1>\c@tocdepth 
  \else
    \par \addpenalty\@secpenalty\addvspace{#2}%
    \begingroup \hyphenpenalty\@M
    \@ifempty{#4}{%
      \@tempdima\csname r@tocindent\number#1\endcsname\relax
    }{%
      \@tempdima#4\relax
    }%
    \parindent\z@ \leftskip#3\relax \advance\leftskip\@tempdima\relax
    \rightskip\@pnumwidth plus4em \parfillskip-\@pnumwidth
    #5\leavevmode\hskip-\@tempdima
      \ifcase #1
       \or\or \hskip 1em \or \hskip 2em \else \hskip 3em \fi%
      #6\nobreak\relax
    \hfill\hbox to\@pnumwidth{\@tocpagenum{#7}}\par
    \nobreak
    \endgroup
  \fi}
\makeatother

\hypersetup{
bookmarks,
bookmarksdepth=3,
bookmarksopen,
bookmarksnumbered,
pdfstartview=FitH,
colorlinks,backref,hyperindex,
linkcolor=Sepia,
anchorcolor=BurntOrange,
citecolor=MidnightBlue,
citecolor=OliveGreen,
filecolor=BlueViolet,
menucolor=Yellow,
urlcolor=OliveGreen
}

\newsavebox{\pullback}
\sbox\pullback{%
\begin{tikzpicture}%
\draw (0,0) -- (1ex,0ex);%
\draw (1ex,0ex) -- (1ex,1ex);%
\end{tikzpicture}}

\newsavebox{\pullbackdl}
\sbox\pullbackdl{%
\begin{tikzpicture}%
\draw (-1ex,0ex) -- (0ex,0ex);%
\draw (0ex,-1ex) -- (0ex,0ex);%
\end{tikzpicture}}

\newsavebox{\pushoutdr}
\sbox\pushoutdr{%
\begin{tikzpicture}%
\draw (-1ex,-1ex) -- (-1ex,0ex);%
\draw (-1ex,0ex) -- (0ex,0ex);%
\end{tikzpicture}}

\newcommand{\Frac}{\mathrm{Frac}}
\newcommand{\Proj}{\mathrm{Proj}}

\DeclareMathOperator{\Spec}{Spec}

\DeclareMathOperator{\Hom}{Hom}
\DeclareMathOperator{\Ext}{Ext}

\DeclareMathOperator{\Pic}{Pic}
\DeclareMathOperator{\Br}{Br}

\DeclareMathOperator{\im}{Im}

\DeclareMathOperator{\Sym}{Sym}
\DeclareMathOperator{\pdeg}{p-deg}

\newcommand*{\coloneq}{\mathrel{\mathop:}=}

\theoremstyle{plain}
\newtheorem{theorem}{Theorem}[section]
\newtheorem*{theorem*}{Theorem}

\newtheorem{proposition}[theorem]{Proposition}
\newtheorem{lemma}[theorem]{Lemma}

\newtheorem*{claim*}{Claim}

\theoremstyle{definition}
\newtheorem{definition}[theorem]{Definition}

\newtheorem{example}[theorem]{Example}
\newtheorem*{setup*}{Setup}

\theoremstyle{remark}
\newtheorem{remark}[theorem]{Remark}

\numberwithin{equation}{theorem}



\newif\ifshowColoursAndTodoes
\showColoursAndTodoestrue

\ifshowColoursAndTodoes
\def\todo#1{\textcolor{Mahogany}%
{\footnotesize\newline{\color{Mahogany}\fbox{\parbox{\textwidth-15pt}{\textbf{todo: } #1}}}\newline}}
\def\commentbox#1{\textcolor{Mahogany}%
{\footnotesize\newline{\color{Mahogany}\fbox{\parbox{\textwidth-15pt}{\textbf{comment: } #1}}}\newline}}

\else
\def\todo#1{}
\def\commentbox#1{}
\renewcommand{\st}[1]{}
\colorlet{red}{black!100} 
\colorlet{teal}{black!100} 
\colorlet{brown}{black!100} 
\colorlet{blue}{black!100} 
\colorlet{magenta}{black!100} 
\colorlet{purple}{black!100} 
\colorlet{cyan}{black!100} 
\fi

\title[Brauer--Manin obstruction from failure of Kodaira Vanishing
]{Brauer--Manin Obstruction from Failure of Kodaira Vanishing in Positive Characteristic
}

\author{Fabio Bernasconi} 
\address{Dipartimento di Matematica “Guido
Castelnuovo”, SAPIENZA Università di Roma, Piazzale Aldo Moro 5, I-00185,
Roma} 
\email{fabio.bernasconi@uniroma1.it}

\author{Domenico Valloni} 
\address{\'Ecole Polytechnique F\'ed\'erale de Lausanne, SB MATH CAG, MA C3 (B\^atiment MA), Station 8, CH-1015 Lausanne, Switzerland}
\email{domenico.valloni@epfl.ch}

\begin{document}

\subjclass[2020]{14G12, 14F22, 14G17, 14F17, 14J20}   

\keywords{Rational points, Brauer--Manin set, Kodaira vanishing theorem, positive characteristic}
\maketitle

\begin{abstract}
We study the Brauer-Manin set of smooth projective surfaces over global fields of positive characteristic for which Kodaira vanishing fails. Our results apply in particular to Raynaud surfaces.
\end{abstract}

\setcounter{tocdepth}{2}

\section{Introduction}
Let $X$ be a geometrically connected smooth projective variety over a global field $K$, and let $\mathbf A_K$ denote the ring of adeles of $K$. A central problem in Diophantine geometry is to understand how the set of rational points $X(K)$ sits in the adelic space $X(\mathbf A_K)$. The Brauer--Manin pairing gives a canonical subset
\[
X(\mathbf A_K)^{\Br(X)}\subseteq X(\mathbf A_K)
\]
which contains the diagonal image of $X(K)$, and whose failure to be all of $X(\mathbf A_K)$ gives an obstruction to the Hasse principle and to weak approximation (for a survey, see \cite[Chapter 13]{Brauer--Groth--book}). 

In this paper we study a phenomenon specific to positive characteristic: we show that a sufficiently strong failure of Kodaira vanishing for surfaces gives rise to an explicit description of the Brauer-Manin set cut-out by the $p$-torsion Brauer group. More precisely,  assume that $K$ is a global function field, and that $X$ descends to a smooth projective surface $\widetilde X$ over $K^p \subset K$. Suppose now that $\widetilde{ \mathcal A}$ is an ample line bundle on $\widetilde X$ for which Kodaira vanishing fails, and that moreover $\widetilde{\mathcal A}^{\otimes p}$ is globally generated and its linear system separates points. The failure of Kodaira vanishing then produces a regular $p$-closed foliation $\widetilde{\mathcal F}$ on $\widetilde X$ and hence a purely inseparable morphism $$ \phi: Y \coloneq \widetilde{X} / \widetilde{\mathcal F}\to X $$ of degree $p$ by Proposition \ref{prop: factorization_frob}. 

\begin{theorem}[See Theorem \ref{thm: criterio_Kodaira_adelic_rational}] \label{thm: intro1}
In the situation above, if $(x_v)\in X(\mathbf A_K)^{\Br(X)[p]}$ does not lift to an adelic point of $Y$, then $(x_v)\in X(K)$.
\end{theorem}
 Since $\phi$ is purely inseparable and finite of degree $p$, the subset $\phi(Y(\mathbf A_K))\subset X(\mathbf A_K)$ is closed and has measure zero. Thus, outside a measure-zero exceptional locus, every adelic point surviving the obstruction is rational.

We then apply this criterion to Raynaud surfaces, originally constructed as the first counterexample to Kodaira vanishing \cite{Ray78}. Although they are surfaces of general type, they exhibit several characteristic-$p$ pathologies: they are uniruled, their Picard schemes are non-reduced, their Hodge-to-de Rham spectral sequence does not degenerate, and they are Shioda supersingular while carrying nonzero global $2$-forms.

\begin{theorem}
There are infinitely many Raynaud surfaces $X$ such that, if $\mathcal A\in \Pic(X)$ is the line bundle exhibiting the failure of Kodaira vanishing, then $\mathcal A^{\otimes p}$ is globally generated and its linear system separates points. In particular, Theorem \ref{thm: intro1} applies when both $X$ and $ \mathcal A$ descend to $K^p \subset K$. 
\end{theorem}

We now place our results in context with the previous literature. 
Explicit descriptions of Brauer--Manin sets over global function fields are known in several important cases, but the existing results are largely restricted to either curves or subvarieties of abelian varieties. For instance, in \cite{PV10} Poonen and Voloch proved that, for varieties of general type which are embedded in an abelian variety over a global function field, the Brauer--Manin obstruction cuts out exactly the rational points when the ambient abelian variety has no isotrivial quotient and a finiteness condition of its $p$-torsion points is satisfied. More recently, Creutz \cite{creutz2025adelicmordelllangbrauermaninobstruction} proved Conjecture C from \cite{PV10}, showing that the assumption on $p$-torsion can be dropped. Curves are studied in \cite{Creutz-Voloch-constant} and \cite{MR5000785}: for constant curves they characterize the complement of the set of rational points in the Brauer-Manin set, whereas for non-isotrivial curves they show that the equality $C(K) = C(\mathbf{A}_K)^{\mathrm{Br}(C)}$ always holds. 

Since Raynaud surfaces are not subvarieties of abelian varieties, as they are uniruled, our results give a genuinely new kind of example for which the Brauer--Manin set can be described explicitly. 
\subsection*{Plan of the article and proofs} We briefly explain the plan of the article and the strategy of the proofs. We start with the basic observation that if $H^{1}(X,\mathcal{L}^{-1})\neq 0$, but $H^1(X,\mathcal{L}^{\otimes -p}) = 0$, then there exists an inclusion $\mathcal{L} \to B^1_X$ in the sheaf of absolute locally exact forms (see Lemma \ref{lem: vanishing_criterion}). This yields by adjunction an inclusion $\mathcal{L}^{\otimes p} \subset \Omega^{1}_{X}$. If now $\mathcal{L}^{\otimes p}$ is globally generated, we can apply \cite[Theorem 1.4]{Val24}, which says that if an adelic point $(x_v) \in X(\mathbf{A}_K)^{\mathrm{Br}(X)[p]}$ satisfies the following condition \[(\star) \text{: there exists a place } v \text{ and } \omega \in H^0(\mathcal{L}^{\otimes p}) \subset H^0(\Omega^1_X) \text{ such that } x_v^* (\omega) \in \Omega^1_{K_v} \text{ is } \neq 0, \] then $\psi_{\mathcal{L}^{\otimes p}}((x_v)_v) \in \mathbb{P}^N(K)$, where $\psi_{\mathcal{L}^{\otimes p}}$ is the morphism associated to the line bundle $\mathcal{L}^{\otimes p}$. Hence, if $\mathcal{L}^{\otimes p}$ separates points as well, it follows that every adelic point in $X(\mathbf{A}_K)^{\mathrm{Br}(X)[p]}$ which also satisfies condition $(\star)$ must be global. One problem is that condition $(\star)$ is not easy to control for a general sub-bundle of $\Omega^1_X$. This is also where the failure of Kodaira vanishing plays a role: since the inclusion $\mathcal{L}^{\otimes p}$ in $\Omega^1_{X}$ comes from an inclusion $\mathcal{L} \subset B^1_{X}$, it induces a $p$-closed foliation $\mathcal{F}$ on $X$ (see Proposition \ref{prop: p-closed-foliations}).
Our main goal in the first part of the article is then to characterize the points satisfying $(\star)$ in terms of the foliation induced by $\mathcal{L}^{\otimes p} \subset \Omega^1_X$ at least when $X$ descends to the Frobenius of $K$. Proposition \ref{prop: equivalence-star-condition} and the consequent Theorem \ref{thm: criterio_Kodaira_adelic_rational} are to be considered the main results in this first part of the article. 

Having developed the above machinery, in Section \ref{sec: Raynaud} we recall the construction of Raynaud surfaces and the explicit ample line bundle $\mathcal{A}$ for which Kodaira vanishing fails. In Proposition \ref{prop: saturation} we describe the foliation associated to $\mathcal{A}^{\otimes p}$ in a purely geometric way. Finally, in Proposition \ref{prop: bpf_lin_system} and then Proposition \ref{prop: examples of good tango structures} we construct infinitely many examples for which $\mathcal{A}^{\otimes p}$ is globally generated and separates points, allowing us to apply Theorem \ref{thm: intro1} in this case. 

\subsection*{Acknowledgement.} We would like to thank T. Kawakami for useful discussions on the Bogomolov--Sommese vanishing theorem.
FB was partially supported by PZ00P2-21610 from the Swiss National Science Foundation, DV by ERC Starting Grant No. 804334.

\section{Failure of Kodaira, foliations and Brauer-Manin obstruction}
In this section we shall explain the connection between the failure of Kodaira vanishing and the Brauer-Manin obstruction.
First, we collect the notation for the article and the basic properties of the Cartier isomorphism.
We then study foliations for varieties over arbitrary fields $K$ which descends to $K^p$ and recall their connection to failure of Kodaira vanishing. The main results are Proposition \ref{prop: equivalence-star-condition} and Theorem \ref{thm: criterio_Kodaira_adelic_rational}, where we draw the connection with the Brauer--Manin set.

\subsection{Conventions and notations}
\begin{enumerate}
    \item We fix $p$ a prime number, and $q=p^e$ be a power of $p$. We denote by $\mathbb{F}_q$ be the finite field with $q$ elements.
    \item We denote by $K$ a field of characteristic $p>0$. Unless stated otherwise, our fields $K$ are all $F$-finite (i.e. $K^p \subset K$ is a finite extension). We denote its $p$-degree as $\pdeg(K) = \log_p[K:K^p]= \dim_K \Omega^1_{K/K^p}$.
    \item A variety over a field $K$ is a geometrically connected separated $K$-scheme of finite type. We say it is a curve (resp. a surface) if it has dimension 1 (resp. 2).
    \item Given a $\mathbb{F}_p$-scheme $X$, we denote by $F_X \colon X \to X$ its absolute Frobenius morphism.
    \item We say $K$ is a \emph{global function field} if $K=\Frac(C)$ is the fraction field  of a regular curve $C$ over $\mathbb{F}_q$, such that $\mathbb{F}_q$ is integrally closed in $K$.
    We say that $K$ is a global field if it is either a number field or a global function field. 
    \item Given a global field $K$, we denote by $\mathbf{A}_K$ the ring of adèles of $K$ equipped with the restricted product topology. 
    An \emph{adelic point} of a $K$-scheme $X$ is an element $x \in X(\mathbf{A}_K)$.
    \item If $X$ is a proper scheme over $K$, we have $X(\mathbf{A}_K)= \prod_{v \text{ place}} X(K_v)$ and an adelic point $x \in X(\mathbf{A}_K)$ is denoted by $(x_v) \in \prod_v X(K_v)$.
    \item Given a variety $X$, we denote by $\Br(X) \coloneqq H^2_{\text{\'et}}(X, \mathbb{G}_m)$ the Brauer group of $X$ and $\Br(X)[p]$ its subgroup of $p$-torsion classes.
    \item We denote by $X(\mathbf{A}_K)^{\Br(X)}$ (resp. $X(\mathbf{A}_K)^{\Br(X)[p]}$) the adelic points of $X$ cut out by Brauer classes (resp. $p$-torsion classes). We refer to \cite[Section 13.3]{Brauer--Groth--book} for an overview on the Brauer--Manin obstruction.
    \item For a locally free sheaf of finite rank $\mathcal{E}$ on a variety $X$, we denote its projectivisation by $\mathbb{P}_X(\mathcal{E})= \Proj_X \Sym^{\bullet} \mathcal{E}^{\vee}$.
    \item Given a globally generated line bundle $\mathcal{L}$ on a $K$-scheme $X$, we denote by $\psi_{\mathcal{L}} \colon X \to \mathbb{P}(H^0(X, \mathcal{L}))$ the natural $K$-morphism associated to $\mathcal{L}$.
    \item Let $X$ be a normal $K$-variety and let $\mathcal{F} \subset \mathcal{E}$ an inclusion of locally free sheaves on $X$. We say that $\mathcal{F} \subset \mathcal{E}$ is \emph{saturated} if the quotient $\mathcal{E}/\mathcal{F}$ is torsion-free.
\end{enumerate}
Since global fields of positive characteristic are imperfect, we briefly review the Cartier isomorphism for K\"{a}hler differential and the theory of $p$-closed foliations in the absolute setting.

\subsubsection*{Cartier isomorphism}
Given a smooth geometrically connected variety $X$ over an $F$-finite field $K$ of dimension $n$, we denote by $\Omega^1_X \coloneqq \Omega^1_{X/\mathbb{F}_p}$ the sheaf of \emph{absolute K\"{a}hler differentials}. Note that it is a coherent sheaf as $K$ is $F$-finite. As $X$ is smooth over $K$, $\Omega^1_X$ is locally free of rank $n+\pdeg(K)$.
If $(\Omega_X^{\bullet}, d)$ denotes the (absolute) de Rham complex of abelian sheaves, then $(F_*\Omega_X^{\bullet}, F_*d)$ is a complex of $\mathcal{O}_X$-modules.
We denote by $Z^i_{X} \coloneqq \ker(F_*d \colon F_* \Omega^i_X \to F_* \Omega^{i+1}_X)$ (resp. $B^i_{X} \coloneqq \text{im}(F_*d \colon F_* \Omega^{i-1}_X \to F_*\Omega^{i}_X)$) the coherent $\mathcal{O}_X$-module of absolute closed forms (resp. locally exact) of degree $i$ inside $F_*\Omega_X^i$. 
For every $i \geq 0$ we have the following short exact sequence of locally free sheaves
    \[ 0 \to Z^i_X \to F_* \Omega_X^i \to B^{i+1}_X \to 0,\]
and the Cartier isomorphism \cite{Car57, Seshadri1958-1959} which gives for every $i$
    \[ 0 \to B^i_X \to Z_X^i \xrightarrow{C} \Omega_X^i \to 0. \]

\subsubsection*{Foliations}

Given a smooth geometrically connected variety $X$ over an $F$-finite field $K$, we denote by $\Theta_X \coloneqq \Theta_{X/\mathbb{F}_p} \cong \mathrm{Hom}(\Omega^1_X, \mathcal{O}_X)$ the absolute tangent sheaf of $X$, which is naturally a sheaf of Lie algebras with the usual Lie bracket operator on derivations.
\begin{definition}
A saturated subsheaf $\mathcal{F} \subset \Theta_X$ is a \emph{foliation} if it is closed under Lie brackets. One says that $\mathcal{F}$ is \emph{$p$-closed} if it is closed under the operation of $p$-power of derivations, and that $\mathcal{F}$ is \emph{regular} if the quotient $\Theta_X/\mathcal{F}$ is locally free.    
\end{definition}

For the theory of foliations on varieties, we refer to \cite{Eke87, Mad16}.
Given a $p$-closed foliation $\mathcal{F} \subset \Theta_X$, one can form the quotient $X/\mathcal{F}$ which is locally given by $\Spec( \{ f \in \mathcal{O}_X \colon D(f) = 0 \text{ for every } D \in \mathcal{F} \})$. 
The morphism $X \rightarrow X/ \mathcal{F}$ is purely inseparable of height one and degree $p^{\mathrm{rank{\mathcal{F}}}}$. 
Moreover, if $\mathcal{F}$ is also regular on $X$, then the quotient $X/ \mathcal{F}$ is a regular scheme. We have a natural subsheaf $\Theta_{X/K} \coloneqq  \mathrm{Hom}(\Omega^1_{X/K}, \mathcal{O}_X)\subset \mathrm{Hom}(\Omega^1_X, \mathcal{O}_X)$ whose local sections are derivations $D \colon \mathcal{O}_X \rightarrow \mathcal{O}_X$ satisfying $D(K) = 0$ (i.e. $K$-linear derivations). 
In the special case where $\mathcal{F} \subset \Theta_{X/K}$, the quotient map $X \rightarrow X/ \mathcal{F}$ is a morphism of $K$-schemes. 

\subsection{Descending varieties, foliations and Condition ($\star$)}
The goal of this section is to translate the condition $(\star)$ in a more geometric terms in the case where $X$ descends to $K^p$ and the line bundle $\mathcal{L} \subset \Omega^1_X$ induces a $p$-closed foliation.
Let us start by discussing the notion of descent to $K^p$.

\begin{definition}
We say that a variety $X/K$ \emph{descends to $K^p$} if there exist a variety $\widetilde{X}/K$ and an isomorphism
$\varphi_X \colon X \xrightarrow{\sim} \widetilde{X} \times_{F_K} \Spec(K).$
We call $\widetilde{X}$ (together the isomorphism $\varphi_X$) a \emph{model} of $X$ over $K^p$. For convenience, we will often omit $\varphi_X$ from the notation of a model. 

Let $X$ and $Y$ be varieties that descend to $K^p$ and fix two models $\widetilde{X}$ and $\widetilde{Y}$. A morphism $f \colon X \to Y$ is said to \emph{descend to $K^p$} if there exist a morphism $\widetilde{f} \colon \widetilde{X} \to \widetilde{Y}$, such that
$f = \widetilde{f} \times_{F_K} K,$
under the natural identifications induced by the $\varphi_X$ and $\varphi_Y$. In this case, we say that $\widetilde{f}$ is a \emph{model} of $f$ over $K^p$.
\end{definition}
It will be convenient to consider the category of descent pairs whose objects are pairs $(X, \widetilde{X})$ where $X$ is a $K$-variety and $\widetilde{X}/K$ is a model of $X$. 
We define a morphism of descent pairs $(X, \widetilde{X}) \rightarrow (Y, \widetilde{Y})$ simply as a morphism $\widetilde{f} \colon \widetilde{X} \rightarrow \widetilde{Y}$ and note that, after performing base-change, one obtains a morphism $f = (\widetilde{f} \times_{F_K}K) \colon X \rightarrow Y$. Moreover, in the case where $K$ is perfect, both choices of $\widetilde{X}$ and the isomorphism are canonically defined.

\begin{example} \label{Example descending Frob}
Assume that $X = \Spec(K[x_1, \cdots, x_n]/ I)$ is an affine $K$-algebra, with $I$ generated by polynomials $f_k = \sum_I a_{k,J} x^J$ (with the standard multi-index notation). 
If $f_k \in K^p[x_1, \cdots, x_n]$ for every $k$, i.e., if $a_{k,J} \in K^p$ for every $k$ and every multi-index $J$, then $X$ descends to $K^p$ and we can choose $\widetilde{X} = \Spec(K[x_1, \cdots, x_n]/ \widetilde{I})$ where $\widetilde{I}$ is generated by $\widetilde{f}_k = \sum_{J} a_{k,J}^{1/p} x^J$.
\end{example}
When working with the subcategory of affine descent pairs $(\Spec(A), \Spec(\widetilde{A}))$, we can always assume the algebras $(A, \widetilde{A})$ to be presented as in Example \ref{Example descending Frob} without loss of generality.

\begin{lemma}
A variety $X/K$ descends to $K^p$ if and only if there is an affine open covering $X=\bigcup_{i \in I} U_i$ such that each $U_i$ descends to $\widetilde{U}_i/ K^p$ and the gluing maps $\varphi_{ij} \colon U_i \cap U_j \to U_i \cap U_j$ descend to $\widetilde{U}_i \cap \widetilde{U}_j$.
We say that $(U_i, \widetilde{U}_i)$ is an affine open covering of $(X, \widetilde{X})$.
\end{lemma}
\begin{proof}
   Suppose $X$ descends to $K^p$  and let $\widetilde{X}$ such that $X \cong \widetilde{X} \times_{F_K} \Spec(K)$. Consider an affine covering $\bigcup \widetilde{U}_i$ with the gluing maps $\widetilde{\varphi}_{ij} \colon \widetilde{U}_i\cap \widetilde{U}_j \to \widetilde{U}_i \cap \widetilde{U}_j$. Then we have $X \cong \bigcup U_i$ via the gluing $\varphi_{ij}=\widetilde{\varphi}_{ij} \times_{F_K} K$.

    If there is such an affine covering $U_i$ with a model $\widetilde{U}_i$ and $\widetilde{\varphi}_{ij}$, we construct $\widetilde{X}$ as the gluing of $\widetilde{U}_i$ via $\widetilde{\varphi}_{ij}$ (note that  $\widetilde{\varphi}_{ij}$ satisfy the cocyle condition). 
    Then the induced morphism $\widetilde{X} \times_{F_K} \Spec(K) \to X$ is an isomorphism.
\end{proof}

Given a model $\widetilde{X}$ of $X$ over $K^p$, we construct the relative Frobenius morphism $F_{\widetilde{X}/K} \colon \widetilde{X} \to X \cong \widetilde{X} \times_{F_K} K$, which is a purely inseparable $K$-morphism which factorises the absolute Frobenius $F_{\widetilde{X}} \colon \widetilde{X} \to \widetilde{X}$ (its existence and uniqueness follows from the universal properties of fibred products).

Now, if $\pi \colon X \rightarrow \Spec(K)$ is a smooth $K$-variety, we have the short exact sequence for sheaves of K\"{a}hler differentials:
\begin{equation} \label{Kahler sequence}
    0 \rightarrow \pi^* \Omega^1_K \rightarrow \Omega^1_X \rightarrow \Omega^1_{X/K} \rightarrow 0.
\end{equation}
A key property that we are going to use in the presence of a model is the following:
\begin{proposition} \label{prop: splitting_Kahler_sequence}
A model $\widetilde{X}$ over $K^p$ of a $K$-variety $X$ determines a splitting of the short exact sequence \eqref{Kahler sequence}. 
\end{proposition}
\begin{proof}
This is well-known result. We give a proof using the previous language of models. 

Let $t_1, \cdots, t_k$ be a $p$-basis of $K$.
Let us first prove the proposition in the affine case. 
Let $X=\Spec(A)$ where $A=K[x_1,\dots, x_n]/(f_1,\dots, f_m)$ with $f_i \in K^p[x_1,\dots, x_n]$. 
On the other hand $\Omega^1_X$ is generated as an $A$-module by $d t_1, \cdots, dt_k$ (which are point-wise linearly indipendent) and $dx_1, \cdots, dx_n$ modulo the $A$-submodule of relations $R$ generated by $df_k = \sum_I d(a_{k,I} x^I) =  \sum_I a_{k,I} d(x^I)$, where in the last equality we used $d(a_{k,I}) = 0$ since $a_{k,I} \in K^p$. Hence $R \subset Adx_1 \oplus A dx_2 \cdots \oplus Adx_n$ and we have a natural decomposition $\Omega^1_{X} \cong \pi^* \Omega^1_K \oplus \Omega^1_{X/K}$.

As for the general case, let $(\Spec(A_i), \Spec(\widetilde{A}_i))$ be an open affine cover of $(X, \widetilde{X})$ such that $(\Spec(A_i), \Spec(\widetilde{A}_i)$ are as in Example \ref{Example descending Frob}.
As the splitting of the sequence \eqref{Kahler sequence} constructed for affine rings is functorial with respect to localisation for such pairs $(A, \widetilde{A})$, it glues to give a splitting of \eqref{Kahler sequence}.
\end{proof}

\begin{remark}
A different model $\widetilde{X}'$ of $X$ over $K^p$ can yield a different splitting. 
\end{remark}

In particular, given a model $\widetilde{X}$ over $K^p$, we have $\Omega^1_{X} \cong \Omega^1_{X/K} \oplus (\Omega^1_K \otimes_K \mathcal{O}_X)$. Similarly, the natural sequence $0 \to \Theta_{X/K} \to \Theta_{X} \to \pi^*\Theta_K \to 0$ splits yielding a decomposition $\Theta_{X} \cong \Theta_{X/K} \oplus (\Theta_{K} \otimes_K \mathcal{O}_X)$. We now discuss descent of foliations.
\begin{definition}
Let $X$ be a smooth $K$-variety with a model $\widetilde{X}$ over $K^p$.
We say that a distribution $\mathcal{F} \subset \Theta_{X/K}$ descends to $\widetilde{X}$ if there exists $\widetilde{ \mathcal{F} }\subset \Theta_{\widetilde{X}/K}$ such that the base-change $\widetilde{\mathcal{F}} \otimes_{F_K} K \to \Theta_{\widetilde{X}} \otimes_{F_K} K$ is naturally isomorphic to $\mathcal{F} \subset \Theta_{X/K}$. 
We say $\widetilde{\mathcal{F}}$ is a model of $\mathcal{F}$.
\end{definition}

We re-write the condition of descent for a distribution in terms of local affine data. Let $(A, \widetilde{A})$ be pairs of $K$-algebras as in Example \ref{Example descending Frob}. 
We say that a derivation $D \colon A \rightarrow A$ \emph{descends} to $\widetilde{A}$ if there is a derivation $\widetilde{D} \colon \widetilde{A} \rightarrow \widetilde{A}$ such that the following diagram commutes: 
\begin{equation} \label{commutative diagram derivations}
\begin{tikzcd}
	A & A \\
	{\widetilde{A}} & {\widetilde{A}}
	\arrow["D", from=1-1, to=1-2]
	\arrow[from=2-1, to=1-1]
	\arrow["{\widetilde{D}}", from=2-1, to=2-2]
	\arrow[from=2-2, to=1-2]
\end{tikzcd},
\end{equation}
where the vertical maps are induced by the Frobenius of $K$. 
Then a distribution $\mathcal{F} \subset \Theta_{\Spec(A)/K}$ descends to $\widetilde{\mathcal{F}}$ on $\widetilde{X}$ if and only if it is locally generated by derivations $D$ which descend to $\widetilde{X}$ in the previous sense. 
\begin{lemma} \label{lem: uniqueness_descent_foliation}
Let $\mathcal{F} \subset\Theta_{X/K}$ be a distribution and let $\widetilde{X}$ be a model of $X$ over $K^p$.
Then there is at most one possible model $\widetilde{\mathcal{F}} \subset \Theta_{\widetilde{X}/K}$ of $\mathcal{F}$.
\end{lemma}
\begin{proof}
The statement of the lemma can be rephrased in terms of derivations on open affine charts $\Spec(A) \subset X$: if $D \colon A \to A$ descends to a derivation $\widetilde{D}$ on $\widetilde{A}$, then the descent $\widetilde{D}$ is unique and $\widetilde{D} \in \Theta_{\widetilde{A}/K} \subset \Theta_{\widetilde{A}}$. 
The uniqueness of $\widetilde{D}$ follows from the injectivity of $F_K \colon \widetilde{A} \rightarrow A$. For the factorisation:  for $t \in K$, we have $F_K(\widetilde{D}(t)) = D(t^p) = 0$. Hence $\widetilde{D}(t)=0$ again thanks to the injectivity of $F_K$. 
\end{proof}

Let us describe a necessary and sufficient condition for the descent of a distribution $\mathcal{F} \subset \Theta_{X/K}$ in local coordinates.

\begin{lemma} \label{lem: explicit_description_descent_der}
Let $(A, \widetilde{A})$ be a pair as in Example \ref{Example descending Frob} and let $D$ be a $K$-linear derivation $D \colon A \rightarrow A$. Write $D = \sum g_i \partial_{x_i}$. The $D$ descends to $\widetilde{A}$ if and only there are $\widetilde{g}_i \in \widetilde{A}$ such that $F_K(\widetilde{g}_i) = g_i$.
\end{lemma}
\begin{proof}
Suppose first that $F_K(\widetilde{g}_i) = g_i$ for every $i$. Then the derivation $D$ clearly descends to $\widetilde{D} = \sum_i \widetilde{g_i} \partial_{x_i}$. 
For the reverse implication: if $D$ descends to $\widetilde{A}$, we have $g_i = D(x_i) = D(F_K(x_i)) = F_K(\widetilde{D}(x_i))$, and we conclude by defining $\widetilde{g}_i=\widetilde{D}(x_i)$.  
\end{proof}

\begin{lemma}
Let $(X, \widetilde{X})$ be a descent pair.
Then $\mathcal{F} \subset \Theta_{X/K}$ is a foliation (resp. $p$-closed) if and only if $\widetilde{\mathcal{F}} \subset \Theta_{\widetilde{X}/K}$ is a foliation (resp $p$-closed). If $X$ is smooth, then $\mathcal{F} \subset \Theta_{X/K}$ is regular if and only if $\widetilde{\mathcal{F}} \subset \Theta_{\widetilde{X}/K}$ is regular.
\end{lemma}
\begin{proof}
All the statements are local on $X$, so we can suppose without loss of generality that $(X, \widetilde{X})$ is $(\Spec(A), \Spec(\widetilde{A}))$.
It is immediate to see from Lemma \ref{lem: uniqueness_descent_foliation} and \ref{lem: explicit_description_descent_der} that  $\widetilde{\mathcal{F}}$ is closed under Lie brackets (resp. under $p$-powers) if and only if $\mathcal{F}$ is so.  

We now check that $\mathcal{F}$ is saturated in $\Theta_{X/K}$ if and only if $\widetilde{\mathcal{F}}$ is saturated in $\Theta_{\widetilde{X}/K}$. 
Let $\widetilde{D} = \sum_i \widetilde{f}_i \partial_{x_i} \in \Theta_{\widetilde{X}/K}$ be a derivation such that $\widetilde{a} \widetilde{D} \in \widetilde{\mathcal{F}}$ for some $\widetilde{a} \in \widetilde{A}$. Then $F_K(\widetilde{a}) D \in \mathcal{F}$ with $D=\sum_i F_K(\widetilde{f}_i) \partial_{x_i}$, showing that $D \in \mathcal{F}$ which implies by definition that $\widetilde{D} \in \widetilde{\mathcal{F}}$. 

We now discuss descent of regularity. 
Up to further localizing $\widetilde{A}$, we can assume that $\mathcal{F}$ is a free sub-vector bundle of $\Theta_{X/K}$ and that it admits a basis $D_1, \cdots, D_r$ which descends to $\widetilde{D}_1, \cdots, \widetilde{D}_r$ on $\widetilde{X}$ as in Lemma \ref{lem: explicit_description_descent_der}, so that $\widetilde{\mathcal{F}} \cong \bigoplus_{i=1}^r \widetilde{A} \widetilde{D}_i \subset \Theta_{\widetilde{A}/K}$ is still a sub-vector bundle. 
\end{proof}

\begin{lemma}
Let $\mathcal{F}$ be a $p$-closed foliation on $X$ that descends to $\widetilde{\mathcal{F}}$ on a model $\widetilde{X}$.
Then $\widetilde{X}/ \widetilde{\mathcal{F}}$ is a model of $X/ \mathcal{F}$ over $K^p$. 
\end{lemma}

\begin{proof}
The statement can be checked on affine charts of $X$, so without loss of generality we can suppose that $(X, \widetilde{X})=(\Spec(A), \Spec(\widetilde{A}))$ is an affine descent pair.
Let $D_1, \dots, D_r$ be local generators of $\mathcal{F}$ with descent data $\widetilde{D_1}, \dots, \widetilde{D_r}$, that are local generators of $\widetilde{\mathcal{F}}$.
Then the natural inclusion $\left\{ \widetilde{f} \in A^{\widetilde{\mathcal{F}}} \mid \widetilde{D}_i(\widetilde{f})=0 \right\} \otimes_K \widetilde{K}  \to A^{\mathcal{F}}$ is an isomorphism, concluding the proof.
\end{proof}

\begin{proposition} \label{prop: factorization_frob}
Let $X$ be a smooth variety over $K$ of dimension $n$.
Let $(X, \widetilde{X})$ be a descent pair over $K^p$. 
Let $\mathcal{F} \subset \Theta_{X/K}$ be a $K$-linear regular $p$-closed foliation of rank $r$ that descends to $\widetilde{\mathcal{F}}$ on $\widetilde{X}$. Then  on $Y \coloneqq \widetilde{X}/\mathcal{\widetilde{\mathcal{F}}}$, there is a regular $p$-closed foliation $\mathcal{G}$ of rank $p^{n-r}$ such that the quotient $\phi \colon Y \to 
    Y/\mathcal{G} \cong X$ and we have a factorisation :
      $$F_{\widetilde{X}/K} \colon \widetilde{X} \xrightarrow[]{\varphi} Y \xrightarrow{\phi} X,$$
      where $\varphi$ is finite of degree $p^r$.
Moreover, there is a factorisation of the relative Frobenius $F_{(\widetilde{X}/\widetilde{\mathcal{F}})/K}$:
$$   F_{(\widetilde{X}/\widetilde{\mathcal{F}})/K} \colon \widetilde{X}/ \widetilde{\mathcal{F}} \xrightarrow{\psi} X \rightarrow X/ \mathcal{F}.$$\end{proposition}

\begin{proof}
    The existence of $\varphi$ and its factorisation of the relative $F_{\widetilde{X}/K}$ is well-known in the literature (see e.g. \cite[Proposition 2.1.4]{Mad16}).  
    On $Y$, we define the new distribution $\mathcal{G} \coloneqq \im(F_{\widetilde{X}/K*} \Theta_{\widetilde{X}/K} \to \Theta_{X/K}) \subset \Theta_{X/K}$, which is a $p$-closed foliation.
    
    We are left to verify the statements on the degree $\varphi$, the smoothness of $\widetilde{X}/\widetilde{\mathcal{F}}$ and the properties of $\mathcal{G}$.
    For this, we can work affine locally and suppose that $(X, \widetilde{X})=(\Spec(A), \Spec(\widetilde{A}))$, where $A$ is a smooth $K$-algebra with a local system of parameters $t_1, \dots, t_n$ and that $\mathcal{F}$ is generated by $\partial_{t_1}, \dots, \partial_{t_r}$. 
    Then $\mathcal{G}$ is generated by $\varphi_* \partial_{t_{r+1}}, \dots, \varphi_* \partial_{t_n}$: it is clearly a regular $p$-closed foliation of rank $p^{n-r}$ on $\Spec(\widetilde{A}^{\mathcal{F}})$.
    Moreover, the ring $(\widetilde{A}^{\mathcal{F}})^{\mathcal{G}}$ is naturally isomorphic to $A$ concluding the proof of the factorisation $F_K= \psi \circ \varphi$. 

    The proof of the factorisation $\widetilde{X}/\widetilde{\mathcal{F}}$ is analogous applied to the regular $p$-closed foliation $\mathcal{G}$.
\end{proof}
We can now state the main result of this section, which is a reformulation of condition $(\star)$ of \cite{Val24} in the case the subsheaf $\mathcal{L} \subset \Omega^1_{X/K}$ induces a $p$-closed foliation.

\begin{theorem} \label{prop: equivalence-star-condition}
 Let $X$ be a smooth projective $K$-variety with model $\widetilde{X}$ over $K^p$, which induces a splitting $\Omega^1_X \cong \Omega^1_{X/K} \oplus (\Omega^1_K \otimes_K \mathcal{O}_X)$. 
 Let $\mathcal{F} \subset \Theta_{X/K}$ be a $p$-closed foliation and assume that $\mathcal{F}$ descends to $\widetilde{\mathcal{F}}$ on $\widetilde{X}$.
 Let $\phi \colon Y \coloneqq \widetilde{X}/ \widetilde{\mathcal{F}} \rightarrow X$ be the $K$-linear morphism constructed in Proposition \ref{prop: factorization_frob}. 
 Let $\mathcal{C} \subset \Omega^1_{X/K}$ be the kernel of the natural bilinear pairing $\Omega^1_{X/K} \otimes \mathcal{F} \rightarrow \mathcal{O}_X$. Assume that $\mathcal{C}$ is globally generated and consider the inclusion $\mathcal{C} \subset \Omega^1_X$ given by the splitting induced by the model $\widetilde X$. For any field extension $K \subset L$ and any point $L$-point $x \colon \Spec(L) \to X$, we have a natural morphism $x^* \colon H^0(X, \mathcal{C}) \otimes L \to \Omega^1_L$. Then, the following conditions are equivalent:
 \begin{enumerate}
     \item $x$  admits a lifting  $\widetilde{x} \colon \Spec(L) \to \widetilde{X}/ \widetilde{\mathcal{F}}$ such that $\phi \circ \widetilde{x}=x$,
     \item for every section $\omega \in H^0(X, \mathcal{C})$,  we have  $x^* \omega = 0.$
 \end{enumerate}
\end{theorem}

\begin{proof}

One direction is easy (and does not need the global generation of $\mathcal{C}$): suppose there exists such a lifting $\widetilde{x}$. As $\psi^*\omega=0$ for any $\omega \in H^0(X,  \mathcal{C})$ by definition of $\mathcal{C}$, we deduce that $x^* \omega =0$.

Let us prove the converse implication. Consider now the diagram 
\[\begin{tikzcd}
	X & {\widetilde{X}} & Y & X \\
	{\Spec(L)} &&& {\Spec(L)}
	\arrow["{F_K}", from=1-1, to=1-2]
	\arrow["f", from=1-2, to=1-3]
	\arrow["\phi", from=1-3, to=1-4]
	\arrow["x", from=2-1, to=1-1]
	\arrow["{F_L}", from=2-1, to=2-4]
	\arrow["x", from=2-4, to=1-4]
\end{tikzcd}\]

where the composition $\phi \circ f \circ F_K = F_X$ is the absolute Frobenius of $X$. Now, it is enough to show that $(F_K)_*(x_* (\Theta^1_L)) \subset \widetilde{\mathcal{F}}$. In fact, in this case, by functoriality we have the following decomposition 
\[\begin{tikzcd}
	X & {\widetilde{X}} & Y & X \\
	{\Spec(L)} && {\Spec(L)/ \Theta^1_L} & {\Spec(L)}
	\arrow["{F_K}", from=1-1, to=1-2]
	\arrow["f", from=1-2, to=1-3]
	\arrow["\phi", from=1-3, to=1-4]
	\arrow["{x}", from=2-1, to=1-1]
	\arrow["{F_L}", from=2-1, to=2-3]
	\arrow["\exists \, \widetilde{x}", from=2-3, to=1-3]
	\arrow["{\mathrm{Id}}", from=2-3, to=2-4]
	\arrow["x", from=2-4, to=1-4]
\end{tikzcd}, \]
which shows that the point $x$ lifts to $Y$. 
Now, the inclusion $(F_K)_*{x}_* (\Theta^1_L) \subset \widetilde{\mathcal{F}}$ is equivalent to prove $((F_K)_*{x}_* (\Theta^1_L), \widetilde{\omega}) = 0$ for every local section $\widetilde{\omega}$ of $\widetilde{\mathcal{C}}= \ker(\widetilde{\mathcal{F}})$. 
As $\mathcal{C}$ is globally generated, it is sufficient to prove $((F_K)_*{x}_* (\Theta^1_L), \widetilde{\omega}) = 0$ for every $\widetilde{\omega} \in H^0(\widetilde{X}, \widetilde{\mathcal{C}})$. Fix $\widetilde{\omega} \in H^0(\widetilde{X}, \widetilde{\mathcal{C}})$. For every $\theta \in \Theta^1_L$, by adjunction we have $((F_K)_* x_* (\theta), \widetilde{\omega}) = (\theta, x^* F_K^*\widetilde{\omega})$. Since $F_K^* (\widetilde{\mathcal{C}}) \subset \mathcal{C}$ holds by functoriality, we conclude that $x^* F_K^*\widetilde{\omega}=0$ by assumption (2). Thus, $((F_K)_*{x}_* (\Theta^1_L), \widetilde{\omega}) = 0$, which concludes the proof.
\end{proof}
\subsection{Failure of Kodaira vanishing and the Brauer--Manin obstruction}
In this section we draw the connection between the failure of Kodaira vanishing theorem and the Brauer--Manin obstruction.

We start by recalling how the failure of Kodaira vanishing theorem for $H^1$ implies the failure of Bogomolov--Sommese vanishing proven in \cite[2.1, Critère, p.178]{Szp79} (see also \cite[Theorem 2.22]{Lan09}). 

\begin{lemma} \label{lem: vanishing_criterion}
    Let $X$ be a smooth proper $K$-variety and let $\mathcal{A}$ be a line bundle on $X$. 
    \begin{enumerate}
        \item If $H^0(X, \mathcal{A}^{-\otimes p})=0$ then $H^0(X, B^1_X \otimes \mathcal{A}^{-1}) \neq 0 \implies H^1(X, \mathcal{A}^{-1}) \neq 0$.
        \item If $H^1(X, \mathcal{A}^{-\otimes p})=0$ then $H^1(X,  \mathcal{A}^{-1}) \neq 0 \implies H^0(X, B_X^1\otimes \mathcal{A}^{-1}) \neq 0$.
    \end{enumerate}
\end{lemma}

\begin{proof}
    Tensoring the short exact sequence
    $ 0 \to \mathcal{O}_X \to  F_{X*} \mathcal{O}_X \to B^1_X \to 0$
    by $\mathcal{A}^{-1}$ and applying the projection formula we obtain
    $$ 0 \to \mathcal{A}^{-1} \to F_{X*} (\mathcal{A}^{-\otimes p}) \to B^1_X \otimes \mathcal{A}^{-1} \to 0. $$ 
    Considering the long exact sequence in cohomology and using that $F$ is an affine morphism, the statements (1) and (2) are immediate. 
\end{proof}
Point (1) says that to produce a counterexample to Kodaira vanishing it is enough to construct a global section of $B^1_X \otimes \mathcal{A}^{-1}$ for an ample line bundle $\mathcal{A}$. 
Using Serre vanishing one can apply Lemma \ref{lem: vanishing_criterion} above for every sufficiently large $p$-power of $\mathcal{A}$:

\begin{proposition} \label{prop: kodaira_implies_subsheaf}
     Let $X$ be a smooth projective variety over $K$.
     If $\mathcal{A}$ is an ample line bundle on $X$ such that $H^1(X, \mathcal{A}^{-1})\neq 0$, then there exists $n \geq 0$, an injective homomorphism $\mathcal{A}^{\otimes p^n} \to B^1_X$ which induces by adjunction $\mathcal{A}^{\otimes  p^{n+1}} \to\Omega_X^1$. 
\end{proposition}

\begin{proof}
    By Serre vanishing, there exists $m > 0$ such that $H^1(X, \mathcal{A}^{-\otimes p^{m}}) \neq 0 $ and $H^1(X, \mathcal{A}^{-\otimes p^{m+1}}) = 0 $.
    By item (1) of Lemma \ref{lem: vanishing_criterion}, there exists a non-zero section of $H^0(X, B^1_X \otimes \mathcal{A}^{-\otimes p^{m}})$.  
    In particular, there exists an injective homomorphism $\mathcal{A}^{\otimes p^m} \to F_{X*}\Omega^1_X$ which in turn gives by adjunction an injective homomorphism $F_X^* \mathcal{A}^{\otimes p^m} \cong \mathcal{A}^{\otimes p^{m+1}}\to \Omega_X^1$.
\end{proof}
For an ample line bundle $\mathcal{A}$, we define $n(\mathcal{A})$ as the minimum $n\geq0$ such that $H^1(X, \mathcal{A}^{- \otimes p^{n}})= 0$. So $\mathcal{A}$ yields a counterexample to Kodaira vanishing in degree 1 if and only if $n(\mathcal{A}) > 0$. We now show that invertible subsheaves of $B^1_X$ naturally define $p$-closed foliations: 

\begin{proposition} \label{prop: p-closed-foliations}
Let $j \colon \mathcal{A} \to B_X^1 \subset F_*\Omega_X^1$ be an invertible subsheaf. The natural morphism  $F^*\mathcal{A} \cong \mathcal{A}^{\otimes p} \subset \Omega_X^1$ given by adjuntion defines a $p$-closed foliation $\mathcal{D} = \ker (\mathcal{A}^{\otimes p}) \subset \Theta_{X}$ of co-rank 1. Moreover if $\mathcal{A}$ is ample, then $\mathcal{A}^{\otimes p} \rightarrow \Omega^1_X \rightarrow \Omega^1_{X/K}$ is injective and hence $\ker(\mathcal{A}^{\otimes p})$ is naturally a $K$-linear foliation.
\end{proposition}
\begin{proof}
We show that $\mathcal{D}$ is closed under Lie brackets and $p$-closed. For this, we can work locally and suppose that $\mathcal{A}$ is generated by a locally exact $1$-form $\omega$ and $\mathcal{D}_X=\ker(\omega)$.
    As $\omega \in B^1_X \subset Z^1_X$, we have $d \omega=0$ and therefore for every $D_1, D_2 \in \mathcal{D}$ we have
    $$0=d\omega(D_1,D_2)=D_1(\omega(D_2))-D_2(\omega(D_1))-\omega([D_1, D_2])=-\omega([D_1,D_2]). $$
    This shows that $\mathcal{D}$ is closed under Lie brackets. Now, as $\omega \in B^1_X$, we have $C(\omega)=0$. By \cite[Proposition 3]{Seshadri1958-1959}, we have for every $D \in \mathcal{D}$ that
    $$ 0=(C(\omega)(D))^p= \omega(D^{[p]})-D^{[p-1]}(\omega(D))=\omega(D^{[p]}), $$
    which shows that $\mathcal{D}$ is closed under the $p$-power operation.
    Thus $\mathcal{D}$ defines a $p$-closed foliation of corank 1. We now prove the remaining claim. Consider the natural composition ${\mathcal{A}}^{\otimes p} \to \Omega^1_{{X}} \to \Omega^1_{{X}/K}$, whose image induces the $K$-linear $p$-closed foliation $\Theta_{X/K} \cap \ker(\mathcal{A}^{\otimes p})$. We are left to show that the composition is injective, and thus that $\ker(\mathcal{A}^{\otimes p}) \subset \Theta_{X/K} $ defines a $K$-linear foliation of co-rank 1. But if it were not injective, since $\mathcal{A}$ is a line bundle, it would factorise through the kernel $\Omega^1_K \otimes \mathcal{O}_X$ of $\Omega^1_X \rightarrow \Omega^1_{X/K}$. But $\Hom(\mathcal{{L}}^{\otimes p}, \Omega^1_K \otimes \mathcal{O}_{{X}})=H^0({X}, \mathcal{{L}}^{\otimes -p} \otimes (\Omega^1_K \otimes \mathcal{O}_{{X}}) )=0$ by ampleness.
\end{proof}

\begin{proposition}[Descending to $K^p$]\label{prop: foliaton_on_descent}
    Let $X$ be a smooth proper $K$-variety, and $\mathcal{A}$ an ample line bundle on $X$ such that $H^1(X, \mathcal{A}^{-1})=0$.
    Suppose that $(X, \mathcal{A})$ descend to $(\widetilde{X}, \mathcal{\widetilde{L}})$.
    Then for $n = n(\mathcal{A})$ there is a natural inclusion $\widetilde{\mathcal{A}}^{\otimes p^{ n+1}} \to \Omega^1_{\widetilde{X}/K}$ which defines a $K$-linear $p$-closed foliation of co-rank 1.
\end{proposition}

\begin{proof}
    By Lemma \ref{prop: kodaira_implies_subsheaf} and Proposition \ref{prop: p-closed-foliations}, there exists a foliation $\widetilde{\mathcal{A}}^{\otimes p+1} \to \Omega^1_{\widetilde{X}}$. 
    Consider the natural composition $\widetilde{\mathcal{A}}^{\otimes p+1} \to \Omega^1_{\widetilde{X}} \to \Omega^1_{\widetilde{X}/K}$, whose image induces a $K$-linear $p$-closed foliation. 
    We are left to show that the composition is injective, and thus it defines a $K$-linear foliation of co-rank 1.
    If it is not injective, as $\mathcal{A}$ is a line bundle, it would factorise through the kernel $\Omega^1_K \otimes \mathcal{O}_X$.
    As $\Hom(\mathcal{\widetilde{L}}^{\otimes p+1}, \Omega^1_K \otimes \mathcal{O}_{\widetilde{X}})=H^0(\widetilde{X}, \mathcal{\widetilde{L}}^{\otimes -p-1} \otimes (\Omega^1_K \otimes \mathcal{O}_{\widetilde{X}}) )=0$ by ampleness of $\mathcal{A}$, we conclude. 
\end{proof}
The following theorem summarize what we have done until now:
\begin{theorem} \label{thm: criterio_Kodaira_adelic_rational}
Let $(X, \mathcal{A})$ be a smooth projective variety together with an ample line bundle over a global function field $K$. Assume that $(X, \mathcal{A})$ descends to $(\widetilde{X}, \widetilde{\mathcal{A}})$ over $K^p$. Assume moreover one of the following hold: 
\begin{enumerate}
    \item [(a)] $H^1(X, \mathcal{A}^{-1}) \neq 0$ and $\mathcal{A}^{\otimes p^n}$ is globally generated for $n = n(\mathcal{A})$;
    \item [(b)] $H^0(X, B^1_X \otimes \mathcal{A}^{-1}) \neq 0$ and $\mathcal{A}^{\otimes p}$ is globally generated.
\end{enumerate} 
Let $Y$ be the quotient of the foliation $\ker(\mathcal{A}^{\otimes p^n})$ in case (a) (resp. $\ker(\mathcal{A}^{\otimes p})$ in case (b)) of $\widetilde{X}$. 
Let 
\[
\phi \colon Y \rightarrow X
\]
be the purely inseparable morphism given by Proposition \ref{prop: factorization_frob}. Finally, let 
\[
\psi_{\mathcal{A}^{\otimes p^n}} \colon X \rightarrow \mathbb{P}^N_K
\]
(resp. by $\psi_{\mathcal{A}^{\otimes p}}$) be the morphism induced by $\mathcal{A}^{\otimes p^n}$ (resp. $\mathcal{A}^{\otimes p}$). Then the following hold:
\begin{enumerate}
    \item[(i)] if $(x_v) \in X(\mathbf{A}_K)^{\mathrm{Br}(X)[p]}$ does not belong to the image of the induced map $Y(\mathbf{A}_K) \rightarrow X(\mathbf{A}_K)$, then $
\psi_{\mathcal{A}^{\otimes p^n}}((x_v)) \in \mathbb{P}^N_K(K)
\quad \text{(resp. } \psi_{\mathcal{A}^{\otimes p}}((x_v)) \in \mathbb{P}^N_K(K)\text{)}.$
 \item[(ii)] If furthermore $\mathcal{A}^{\otimes p^n}$ (resp. $\mathcal{A}^{\otimes p}$)  separates points, then $(x_v) \in X(K)$ and therefore
\[
\bigl(X(\mathbf{A}_K) \setminus \phi({Y}(\mathbf{A}_K))\bigr)^{\mathrm{Br}(X)[p]} = X(K) \setminus \phi({Y}(K)).
\]
\end{enumerate}
\end{theorem}

\begin{proof}
    Note that $\mathcal{A}^{\otimes p} \subset \Omega^1_{\widetilde{X}}$ descends to $\widetilde{X}$ and it defines $\widetilde{\mathcal{F}}$ which is a $p$-closed foliation by Proposition \ref{prop: p-closed-foliations}.
    Then we can combine Proposition \ref{prop: equivalence-star-condition} with \cite[Theorem 1.4]{Val24} to conclude the statement.
\end{proof}

\begin{remark}
Note that the variety $Y$ constructed in Theorem \ref{thm: criterio_Kodaira_adelic_rational} may not be regular. 
\end{remark}
\section{Raynaud surfaces} \label{sec: Raynaud}

We now aim to apply the previous theory to the case of Raynaud surfaces over a global field $K$.
We will show that there are Raynaud surfaces $X$ with an ample line bundle $\mathcal{A}$ that violates Kodaira vanishing for which $\mathcal{A}^{\otimes p}$ separates points. We start by recalling the construction of Raynaud surfaces starting from Tango structures on curves that descend to $K^p$. We then describe geometrically the failure of Kodaira vanishing for an explicit ample line bundle associated to a special $p$-closed foliation on $X$. We finally prove Theorem \ref{main thm raynaud end paper} on the Brauer--Manin set of Raynaud surfaces.

\subsection{Tango curves} \label{subsection: tango}
Tango curves were introduced in \cites{Tan72} to study the behavior of extension of vector bundles after the pullback by the Frobenius morphism. 
Here we recall their definition.
Let $C$ be a smooth geometrically connected projective curve over $K$ and let $\mathcal{L} \in \Pic(C)$ be an ample line bundle.
Assume that we can find $0 \neq \alpha \in H^1(C, \mathcal{L}^{-1})$ such that $F_C^*(\alpha)=0$. We say $(\mathcal{L}, \alpha)$ is a \emph{pre-Tango structure} on $C$, which can be characterised via differential forms as follows:
\begin{lemma}
A pre-Tango structure $(\mathcal{L}, \alpha)$ is equivalent to a nontrivial homomorphism $\varphi_{\alpha} \colon \mathcal{L} \to B^1_{C}$. If moreover $\mathcal{L} \cong  \mathcal{O}_C(D)$ where $D$ is an effective divisor, then $\alpha$ corresponds to the datum of a locally exact 1-form $\omega_{\alpha} \in H^0(X, B_{C}^1(-D))$ with zeroes along $D$.
\end{lemma}
\begin{proof}
This is just a reformulation of Lemma \ref{lem: vanishing_criterion}.
\end{proof}
The composition $\mathcal{L} \xrightarrow{\varphi_{\alpha}} B^1_{C} \to (F_C)_*\Omega^1_{C}$ gives rise by adjunction and composition to a homomorphism $\psi_{\alpha} \colon \mathcal{L}^{\otimes p} \to \Omega^1_{C} \to \Omega^1_{C/K}$. This is non-trivial because $\mathcal{L}$ is ample as in Proposition \ref{prop: p-closed-foliations}. 
This shows in particular that $p \deg(\mathcal{L}) \leq 2g(C)-2$. 
\begin{definition}
    We say a pre-Tango structure $(\mathcal{L}, \alpha)$ on $C$ is \emph{Tango} if $\psi_{\alpha} \colon \mathcal{L}^{\otimes p} \to \Omega^1_{C/K}$ is an isomorphism (and thus $p \deg(\mathcal{L}) = 2g(C) - 2$). 
    Given an integer $d>0$, a triple $(\mathcal{L}, \mathcal{N}, \alpha)$ is a \emph{Tango structure of index $d$} on $C$ if $(\mathcal{L}, \alpha)$ is a Tango structure and $\mathcal{N}$ is a line bundle such that $\mathcal{L} \cong {\mathcal{N}}^{\otimes d}$.
\end{definition}

To remain concrete, we now give some examples of Tango structures.
\begin{example}[Planar curves, {\cite[Example 1.3]{Muk13}}]
    If $p \geq 3$, the classical example of Tango structure of index $d$ is constructed on the planar smooth curve of degree $dp$:
    $$C=\left\{ Y^{dp}-YZ^{dp-1}-ZX^{dp-1}=0 \right\} \subset \mathbb{P}^2_{K}.$$
    Note that $C \cap \left\{Z=0\right\} = dp [1:0:0]$.  Denote $Q=[1:0:0]$.
    On the affine chart $Z \neq 0$, $C \setminus Q$ is described by the Artin--Schreier type equation $(y^{dp}-y-x^{dp-1}=0)$.
    The differential $dx$ has no zeroes nor poles on $C \setminus Q$ and thus it defines an exact global 1-form $d\left(\frac{X}{Z} \right) \in H^0(C, \omega_{C})$ with associated divisor $dp(dp-3)Q$.
    In particular, it defines a Tango structure $\alpha$ on $\mathcal{L}=\mathcal{O}_C(d(dp-3)Q)$ and $(   \mathcal{L}, \mathcal{N}=\mathcal{O}_C(dp-3), \alpha)$ is a Tango structure of index $d$. 
\end{example}

\begin{example}[Covers of supersingular elliptic curves] 
We give another example constructed geometrically which does not seem to appear in the literature. 
Let $E$ be a geometrically connected supersingular smooth genus 1 curve over a perfect field $K$. As $E$ is supersingular, there exists a class $\omega \in H^0(E, B^1_E) \neq 0$. Let $\mathcal{A}$ be an ample line bundle on $E$ of degree $a$, set $k = dp +1$ and let $s \in H^0(C, \mathcal{A}^{\otimes k})$ be a section such that $Z(s)= \sum [ q_i ] $ is reduced. Let $\pi \colon C \to E$ be the $k$-th cyclic covering totally ramified along $Z(s)$ and let $\widetilde{q}_i$ be the unique point over $q_i$. Then $\text{div}(\pi^*\omega) = dp\sum [\widetilde{q}_i]$, thus $\mathcal{L}=\mathcal{O}_C(d(\sum [\widetilde{q}_i)])$ embeds in $B^1_X$. In particular, $\mathcal{N}=\mathcal{O}_C(\sum [\widetilde{q}_i])$ gives a Tango structure of index $d$.
\end{example}

\begin{example}
Other examples of Tango structures can be constructed via hyperelliptic curves, see \cite[Section 3]{TY02}.
\end{example}

Given a model $\widetilde{C}$ of $C$ over $K^p$, we say a Tango structure $(\mathcal{L}, \mathcal{N}, \alpha)$ of index $d$ \emph{descends} to $\widetilde{C}$ if there exists a triple $(\widetilde{\mathcal{L}}, \widetilde{\mathcal{N}}, \widetilde{\alpha})$ on $\widetilde{C}$ such that its base change after Frobenius is isomorphic to the original Tango structure. 

\subsection{Raynaud surfaces}
We briefly recall the construction of Raynaud surfaces (for more in depth discussion, we refer to \cite{Ray78, Tak91, Tak92, TY02, Muk13}).
Let $C$ be a geometrically connected smooth curve over $K$ with a pre-Tango structure $(\mathcal{L},  \alpha)$.
As $H^1(C, \mathcal{L}^{-1}) \cong \Ext^1( \mathcal{O}_C, \mathcal{L}^{-1})\cong \Ext^1( \mathcal{L}, \mathcal{O}_C)$, to the class $\alpha$ one associates a vector bundle $\mathcal{E}$ which fits in a short exact sequence 
$$ 0 \rightarrow \mathcal{O}_C \rightarrow \mathcal{E} \rightarrow \mathcal{L}  \rightarrow 0,$$
and we consider the associated ruled surface $g \colon \mathbb{P}_C(\mathcal{E}) \rightarrow C$.
The inclusion $\mathcal{O}_C \subset \mathcal{E} $ gives a section $S \subset \mathbb{P}_C(\mathcal{E})$ of $g$ with ample normal bundle $N_{S/\mathbb{P}_C(\mathcal{E})} \simeq \mathcal{L}$.
On $\mathbb{P}_C(F^*\mathcal{E})$, the sequence $0 \to \mathcal{O}_C \to F^*\mathcal{E} \to \mathcal{L}^{\otimes p} \to 0$ splits. Using the splitting and tensoring by $\mathcal{L}^{\otimes -p}$ we have the short exact sequence 
$$ 0 \to \mathcal{O}_C \to F^*\mathcal{E} \otimes \mathcal{L}^{-\otimes p} \to \mathcal{L}^{-\otimes p} \to 0 $$
corresponding to a new section $\Gamma' \subset \mathbb{P}_C(F^*\mathcal{E})$ with anti-ample normal bundle $N_{\Gamma'/\mathbb{P}_C(F^*\mathcal{E})} \simeq \mathcal{L}^{- \otimes p}$. Let $F_{\mathcal{E}} \colon \mathbb{P}_C(\mathcal{E}) \to \mathbb{P}_C(F^*\mathcal{E})$  be the morphism given by the fibre product:

 \[ \begin{tikzcd}
   \mathbb{P}_C(\mathcal{E}) \arrow[dr, dotted, "F_\mathcal{E}"] \arrow[drr, bend left, "F_{\mathbb{P}_C(\mathcal{E})}"] \arrow[ddr, bend right, "g"]  & & & \\
   & \mathbb{P}_C(\mathcal{O}_C \oplus \mathcal{L}^{-\otimes p}) \cong {\mathbb{P}_{C}(F^*\mathcal{E}) }\arrow[d, "g'"] \arrow[r]  & {\mathbb{P}_C(\mathcal{E})} \arrow[d, "g"] \\
   & {C} \arrow[r, "F_C"] & {C}
\end{tikzcd} \] and consider the schematic pre-image $\Gamma \coloneqq F_\mathcal{E}^{-1}(\Gamma') \subset  \mathbb{P}_C(\mathcal{E})$ so that the map $\Gamma \rightarrow C$ is purely inseparable of degree $p$ and $\Gamma \cap S = \emptyset$.
We recall that $\Pic(\mathbb{P}_C(\mathcal{E})) \cong g^* \Pic(C) \oplus \mathbb{Z}[\mathcal{O}_{\mathbb{P}(\mathcal{E})}(1)].$
\begin{lemma}
In the previous situation, we have that the following formulas hold:
\begin{enumerate}
    \item $\mathcal{O}_{\mathbb{P}(\mathcal{E})}(S) = \mathcal{O}_{\mathbb{P}(\mathcal{E})}(1)$ ;
    \item $\mathcal{O}_{\mathbb{P}(\mathcal{E})}(\Gamma) = \mathcal{O}_{\mathbb{P}(\mathcal{E})}(p) \otimes g^*(\mathcal{L}^{-\otimes p})$;
    \item $\omega_{\mathbb{P}(\mathcal{E})} \cong \mathcal{O}_{\mathbb{P}(\mathcal{E})}(-2) \otimes g^*{\mathcal{L}^{\otimes (p+1)}}$
\end{enumerate}
\end{lemma}
\begin{proof}
\begin{enumerate}
    \item This is the standard formula of the normal bundle of a section of $\mathbb{P}(\mathcal{E})$. Indeed, by \cite[Proposition III.7.12]{Har77} we have $\mathcal{O}_{\mathbb{P}(\mathcal{E})}(1)|_S  \cong \mathcal{L}$ and, as $\mathcal{O}_{\mathbb{P}(\mathcal{E})}(S)|_S \cong N_{S/\mathbb{P}_C(\mathcal{E})} \cong \mathcal{L}$.
    \item Still using this formula, we compute $\mathcal{O}_{\mathbb{P}(F^*(\mathcal{E}))}(\Gamma')= \mathcal{O}_{\mathbb{P}(F^*(\mathcal{E}))}(1) \otimes (g')^*(\mathcal{L}^{-p})$. We now pull back this via the relative Frobenius $F_{\mathcal{E}}$ and use the commutative diagram above to get $\mathcal{O}_{\mathbb{P}(\mathcal{E})}(\Gamma) =  F_{\mathcal{E}}^*\mathcal{O}_{\mathbb{P}(F^*(\mathcal{E}))}(\Gamma') =  \mathcal{O}_{\mathbb{P}(\mathcal{E})}(p) \otimes F_{\mathcal{E}}^*(g')^*(\mathcal{L}^{-p}) =  \mathcal{O}_{\mathbb{P}(\mathcal{E})}(p) \otimes g^*(\mathcal{L}^{-p})$.
\item Finally, since $\omega_{\mathbb{P}(\mathcal{E})} \cdot F =-2$ for any closed fibre, we have $\omega_{\mathbb{P}(\mathcal{E})} \cong \mathcal{O}_{\mathbb{P}(\mathcal{E})}(-2) \otimes g^*\mathcal{M}$ for some line bundle $\mathcal{M}$.
    As $\mathcal{L}^{\otimes p} \cong \omega_S \cong \omega_S(S)|_S \cong \mathcal{O}(-1)|_S \otimes \mathcal{M} \cong \mathcal{L}^{-1} \otimes \mathcal{M}$, which proves the claim. 
\end{enumerate}
\end{proof}
The Tango condition is reflected by the singularities of the curve $\Gamma$:
\begin{proposition} \label{prop: regularityTango}
    $\Gamma$ is regular if and only if $(\mathcal{L}, \alpha)$ is a Tango structure.
\end{proposition}
\begin{proof}
    See \cite[Proposition 1.7]{Muk13}.
\end{proof}
Assume $(\mathcal{L}, \mathcal{N}, \alpha)$ is a Tango structure of index $d$ (i.e. $\mathcal{N}^{\otimes d} \cong \mathcal{L}$) and assume $d$ divides $p+1$.
As we have 
$$\mathcal{O}_{\mathbb{P}(\mathcal{E})}(S + \Gamma) = \mathcal{O}_{\mathbb{P}(\mathcal{E})}(p+1) \otimes g^*(\mathcal{N}^{-\otimes dp}) = \bigg( \mathcal{O}_{\mathbb{P}(\mathcal{E})} \bigg(\frac{p+1}{d} \bigg) \otimes g^*(\mathcal{N}^{-\otimes p}) \bigg)^{\otimes d} \eqqcolon \mathcal{M}^{\otimes d},$$
we can construct the $d$-th cyclic covering $\pi$ which ramifies over the regular divisor $S + \Gamma$ (regularity follows by Proposition \ref{prop: regularityTango}) and sits in the following commutative diagram:
 \[ \begin{tikzcd}
    {X \coloneqq \mathbb{P}_C(\mathcal{E})[\sqrt[d]{S+\Gamma}]} \arrow[dr, "f"] \arrow[r, "\pi"]  & {\mathbb{P}_C(\mathcal{E})} \arrow[d, "g"] \\
     & C,
\end{tikzcd} \]
where $X \subset \Spec_{\mathbb{P}_C(\mathcal{E})} \Sym^{\bullet}\mathcal{M^{\vee}}$ defined by $(z^d-\pi^*(s\gamma)=0)$, where $z$ is the tautological section and $s$ (resp. $\gamma$) are the section corresponding to $S$ (resp. $\Gamma$).
The surface $X$ is the \emph{Raynaud surface} associated to the Tango structure $(\mathcal{L},\mathcal{N}, \alpha)$ of index $d$  on $C$.
We denote by $\Sigma \coloneqq \pi^{-1}(\Gamma)$ the multi-section of the cusps of $f$ and $T \coloneqq \pi^{-1}(S)$ a section of $f$.
\begin{lemma}
    The following hold:
    \begin{enumerate}
        \item for the normal bundles of $\Sigma$ and $T$, we have 
        $N_{\Sigma/X} \cong \mathcal{N}^{-\otimes p}$ and $N_{T/X} \cong \mathcal{N}$;
        \item $\omega_X \cong (\mathcal{O}_X(pd-p-d-1)T) \otimes f^*\mathcal{N}^{\otimes (d+p)}; $ 
        \item if $F$ is a closed fibre of $f$, $2p_a(F)-2=\deg K_F=dp-p-d-1.$
    \end{enumerate}
\end{lemma}

\begin{proof}
    The first item follows from cyclic covering construction.
    As we know that $\Sigma \sim pT - f^*\mathcal{N}^{\otimes p}$, we have for the canonical divisor:
    $$\omega_X \cong \pi^*\omega_{\mathbb{P}_C(\mathcal{E})}\otimes \mathcal{O}_X((d-1)(\Sigma+T)) \cong f^*{\mathcal{L}^{\otimes {p+1}} \otimes \mathcal{O}_X((d-1)\Sigma-(d+1)T}) $$
    $$
    \cong f^*\mathcal{L}^{p+1} \otimes f^*\mathcal{N}^{\otimes p-dp}\otimes \mathcal{O}_X((d-1)pT-(d+1)T) \cong \mathcal{O}_X((pd-p-d-1)T)\otimes f^*(\mathcal{L} \otimes \mathcal{N}^{\otimes p}).$$
    
    (3) follows from (2) and the adjunction formula.
\end{proof}

\begin{lemma} \label{lem: descent_raynaud}
    Let $\widetilde{C}$ be a model of $C$ over $K^p$ and suppose that the Tango structure of index $d$ $(\mathcal{L}, \mathcal{N}, \alpha)$ on $C$ descends to $\widetilde{C}$.
    Then the Raynaud surface $X$ descends to a Raynaud surface $\widetilde{X}$ over $K^p$.
\end{lemma}

\begin{proof}
    This is immediate by taking base change of the previous construction on $\widetilde{C}$ with the descended Tango structure $(\widetilde{\mathcal{L}}, \widetilde{\mathcal{N}}, \widetilde{\alpha})$.
\end{proof}

\subsection{Failure of Kodaira vanishing for Raynaud surfaces}

In this section we construct the ample line bundle $\mathcal{A} \in \Pic(X)$ which provides the counterexample to Kodaira vanishing for Raynaud surfaces. 
In fact, we show that such line bundle can be understood geometrically as follows: since $X$ is a smooth surface, the fact that the fibration $X \rightarrow C$ is not smooth along a divisor implies that $f^* \Omega^1_{C/K} \subset \Omega^1_{X/K}$ is not a saturated subsheaf. 
The special feature of characteristic $p>0$ is that such a divisor can be horizontal. 

Let us describe better the fibration $f \colon X \to C$ over an affine open set $U \subset C$ for which $\mathcal{E}|_U$ is trivial. 
Let $\omega \in H^0(C, B^1_C \otimes \mathcal{N}^{-\otimes d}) \neq 0$ be the locally exact form given by the Tango structure. 
By construction, $\omega=d \xi$ for a regular function $\xi \in \mathcal{O}_C(U)$. 
In this case, there is a trivialisation $\mathbb{P}_C(\mathcal{E})|_U \cong U \times \mathbb{P}^1_{[y:s]}$ for which the equation of $S$ is given by $(s=0)$ and the equation of $\Gamma$ is given by $(y^p-(\pi^*\xi)s^{p}=0)$.
In particular, the Raynaud surface is locally cut over $U$ by the equation 
$$ (z^d-s(y^p-(\pi^*\xi)s^p)=0)\subset U \times {\mathbb{P}^1}_{[y:s]}. $$
With this description, we can study the map $f^*\Omega^1_{C/K} \to \Omega^1_{X/K}$.
\begin{proposition} \label{prop: saturation}
If $d \geq 2$ the subsheaf $f^* \Omega^1_{C/K} \subset \Omega^1_{X/K}$ is not saturated, and its saturation is $f^* \Omega^1_{C/K}((d-1)\Sigma)$.
Moreover, we have a short exact sequence 
$$ 0 \to  f^*\Omega^1_{C/K}((d-1)\Sigma) \to \Omega^1_{X/K} \to (\omega_X \otimes f^*\omega_C^{-1}) \otimes \mathcal{O}_X(-(d-1)\Sigma) \to 0.$$ 
\end{proposition}
\begin{proof}
To show that the natural homomorphism $f^*\Omega^1_{C/K} \to \Omega^1_{X/K}$ extends to a saturated inclusion $f^*\Omega^1_{C/K}((d-1)\Sigma) \to \Omega^1_{X/K}$, it is enough to perform a local computation about the locus of cusps of $f$. 
Locally \'etale, the fibration $g \colon \mathbb{P}(\mathcal{E}) \to C$ is the projection $\pi_x \colon \mathbb{A}^2_{x,y} \to \mathbb{A}^1_x$. As $\Gamma$ is regular and $\pi_x|_\Gamma$ is purely inseparable of degree $p$, up to possibly passing to a further \'etale cover we can assume $\Gamma$ is described as the vanishing locus of $(y^p+x)$.
Thus locally around the cusps of $X$ we can write $X \cong \Spec k[x,y,z]/(z^d=y^p+x)$ and $f$ is simply the projection onto the $x$ coordinate.
A local computation shows that $f^*(dx)=dz^{d-1}dz$, hence the saturation of $f^*\Omega^1_{C/K} \to \Omega^1_{X/K}$ is given exactly by $f^*\Omega^1_{C/K}((d-1)\Sigma) \to \Omega^1_{X/K}$.

To prove the validity of the short exact sequence, we only need to show that the quotient of the injection is locally free. This follows from the local computations around the cusp done above.
\end{proof}

Using the previous computation, we deduce the following isomorphism
$$f^*\Omega^1_{C/K}((d-1)\Sigma) \cong f^*\mathcal{N}^{\otimes pd} \otimes \mathcal{O}_X((d-1) \Sigma) \cong f^*(\mathcal{N}^{\otimes pd}) \otimes \mathcal{O}_X(p(d-1) (T - f^* \mathcal{N})),$$ and thus $f^*\Omega^1_{C/K}((d-1)\Sigma)$ is divisible by $p$ in the Picard group of $X$. 
We define the $p$-th root line bundle $\mathcal{A}$ as
$$\mathcal{A} \coloneqq f^*(\mathcal{N}^{\otimes d}) \otimes   \mathcal{O}_X((d-1)(T - f^* \mathcal{N})) = f^*(\mathcal{N}) \otimes \mathcal{O}_X((d-1) T). $$
We now show that $\mathcal{A}$ does not satisfy Kodaira vanishing by producing a non-zero section $dz \in H^0(X,B^1_X \otimes \mathcal{A}^{-1})$.

\begin{proposition} \label{prop: foliationpclosed-kodaira}
The line bundle $\mathcal{A}$ is ample and it satisfies the following:
\begin{enumerate}
    \item there is a non-zero class $\alpha$ in $H^1(X, \mathcal{A}^{-1})$ such that $F_X^*\alpha=0$;
    \item the class $\alpha$ is induced by an element in $H^0(X, B^1_X \otimes \mathcal{A}^{-1})$. 
    The induced inclusion $\mathcal{A} \to B^1_X \subset F_*\Omega^1_X$ induces a $p$-closed foliation $\mathcal{A}^{\otimes p} \to \Omega^1_X \to \Omega^1_{X/K}$ which coincides with the saturation $f^*\Omega^1_{C/K}((d-1)\Sigma) \to \Omega^1_{X/K}$ of Proposition \ref{prop: saturation}. 
\end{enumerate}
\end{proposition}
\begin{proof}
Ampleness follows from the Nakai--Moishezon criterion. We prove (2), as item (1) follows from (2) and Lemma \ref{lem: vanishing_criterion}. Let $\omega \in H^0(C, B^1_C \otimes \mathcal{N}^{-\otimes d}) \neq 0$ be the locally exact form given by the Tango structure. 
As $f$ is separable, the natural map 
$H^0(X, f^*B^1_C \otimes f^*\mathcal{N}^{-\otimes d}) \to H^0(X, B^1_X \otimes f^{*}\mathcal{N}^{-\otimes d})$
is injective and we consider 
$0 \neq f^* \omega \in H^0(X, B^1_X \otimes f^{*}\mathcal{N}^{-\otimes d})$.
Taking differentials, we have the relation $dz^{d-1}dz=f^*\omega$ in the chart where $(s\neq 0)$. By saturating as in Proposition \ref{prop: saturation} we note that $dz$ is regular in $U$ and looking at the other chart one sees it vanishes at order $p(d-1)$ along the section $T$. Therefore we deduce that $dz$ is the desired section in $\alpha \in H^0(X, B^1_X\otimes \mathcal{A}^{-1})$. The fact that $\ker(dz)=f^*\Omega^1_{C/K}((d-1)\Sigma)$ is $p$-closed follows from Proposition \ref{prop: p-closed-foliations}.
\end{proof}

\begin{remark} \label{rem: descent_ampleviolatingKodaira}
    Let $\widetilde{C}$ be a model of $C$ over $K^p$ and suppose that the Tango structure of index $d$ $(\mathcal{L}, \mathcal{N}, \alpha)$ on $C$ descends to $\widetilde{C}$. 
    Note that the line bundle $\mathcal{A}$ on the Raynaud surface $X$ descends to a line bundle $\widetilde{\mathcal{A}}$ on $\widetilde{X}$ over $K^p$ described in Lemma \ref{lem: descent_raynaud}.
\end{remark}

We now give an explicit description of the quotient $X / \mathcal{F}$.  
Note that, as the foliation is regular by construction, the quotient of the foliation is a smooth proper surface (and also projective). 
Moreover, by construction the quotient of the foliation fits in the natural commutative diagram 
\[\begin{tikzcd}
	X && {X/ \mathcal{F}} \\
	{\mathbb{P}_C({\mathcal{E}})} && {\mathbb{P}_C(F^*{\mathcal{E}}) \cong \mathbb{P}_C(\mathcal{O}_C \oplus \mathcal{L}^{\otimes p})}
	\arrow["q", from=1-1, to=1-3]
	\arrow["\pi", from=1-1, to=2-1]
	\arrow["\pi_{\mathcal{F}}", from=1-3, to=2-3]
	\arrow["{F_{\mathcal{E}}}", from=2-1, to=2-3].
\end{tikzcd}\]


\subsection{Applications to the Brauer--Manin set of Raynaud surfaces}

In order to apply Theorem \ref{thm: criterio_Kodaira_adelic_rational} to the study of the Brauer--Manin set for Raynaud surfaces, we need first to find sufficient conditions for  $\mathcal{A}^{\otimes p}$ to be globally generated.

\begin{proposition} \label{prop: bpf_lin_system}
If $\mathcal{N}^{\otimes p}$ is globally generated, then the line bundle $\mathcal{A}^{\otimes p}$ is globally generated.
If $C$ is not hyperelliptic and $d=2$, then the morphism $\psi_{\mathcal{A}^{\otimes p}}$ separates points and the it restricts to a purely inseparable morphism of degree $p$ on each fibre of $f$.
\end{proposition}

\begin{proof}
    We denote by $N$ an effective divisor such that $\mathcal{O}_C(N) \cong \mathcal{N}$, and recall we have $pA \sim p(d-1)T+f^*pN \sim (d-1)\Sigma+ f^*pdN$. 
    To prove the statement, we can base change to the algebraic closure of $K$. 
    For the global generation, we show that for every closed point $x \in X$ there exists a section $s \in H^0(X, \mathcal{A}^{\otimes p})$ not vanishing at $x$.
    As $f^*pN$ is globally generated and $T \cap \Sigma= \emptyset$, this is immediate.
    
    As the restriction $\psi_{\mathcal{A}^{\otimes p}}|_T$ coincides with $\psi_{pdN}=\psi_C$ is an embedding as $C$ is not hyperelliptic, to show the latter statement it is sufficient to check that for every closed fibre $F$, the morphism $\psi_{\mathcal{A}^{\otimes p}}|_F$ is purely inseparable.
    
    As $pN$ is globally generated, we can suppose that $f(x) \notin pN$.
    Consider the two dimensional subspace $V$ of $H^0(X, \mathcal{A}^{\otimes p})$ generated by $f^*pN+ p(d-1)T$ and $f^*pdN+(d-1)\Sigma$.
    Let $V_F$ be the image of $V$ by the natural projection morphism $H^0(X, \mathcal{A}^{\otimes p}) \to H^0(F, \mathcal{A}^{\otimes p})$.
    As $\Sigma \cap T =\emptyset$, the (non-complete) linear system $V_F$ is basepoint free.
    Let $\psi_{V_F} \colon F \to \mathbb{P}(V)\simeq \mathbb{P}^1$ be the morphism of degree $p(d-1)$ to $\mathbb{P}^1$. So the degree of $\psi_{V_F}$ is $p$ as $d=2$. In this case, $\psi_{V_F}$ has a zero and a pole of multiplicity $p$, hence it is purely inseparable.
\end{proof}

To show that the conditions of Proposition \ref{prop: bpf_lin_system} are non-empty, we now construct Tango structures of index $d=2$ satisfying those. 

\begin{proposition} \label{prop: examples of good tango structures}
Suppose $p>2$. 
Then there exist non-hyperelliptic curves $C$ with Tango structures $(\mathcal{L}, \mathcal{N}, \alpha)$ of index $2$ (so $\omega_C  \cong \mathcal{N}^{\otimes 2p}$) and theta characteristic $\mathcal{N}^{\otimes p}$ globally generated. 
\end{proposition}
\begin{proof}
Let $D$ be an hyperelliptic curve of odd genus with a Tango structure $(\mathcal{L}, \alpha)$ (such curves exist by \cite{TY02}).
If $\psi_{\omega_X} \colon D \to \mathbb{P}^1_K \hookrightarrow \mathbb{P}^{g-1}_K$ is the canonical map, we have $\psi^*\mathcal{O}_{\mathbb{P}^1_K}(g-1)=\omega_X$. As $g$ is odd, we consider the globally generated theta characteristic $\theta_D \coloneqq \psi^*\mathcal{O}_{\mathbb{P}^1}(\frac{g-1}{2})$.
Let $\pi \colon C \to D$ be a sufficiently high \'etale cover such that $C$ is not hyperelliptic.
As $\pi$ is \'etale, we have that $\omega_C \cong \pi^*\omega_D=\pi^*\mathcal{L}^{\otimes p}$ and that $\theta \coloneqq \pi^*\theta_D$ is a theta characteristic of $C$ and $\pi^*\mathcal{L} \to \pi^*B_D^1 \cong B_C^1$ is a non-zero morphism. 
In particular, $(\pi^*\mathcal{L}, \pi^*\alpha)$ is a Tango structure on $C$.
Let $\mathcal{N}$ be a line bundle such that $\mathcal{N}^{\otimes 2} \cong \mathcal{L}$ such that $\mathcal{N}^{\otimes p} \cong \theta$. Note that this is possible because the set of roots of $\mathcal{L}$ is in bijection with the set of roots of $\omega_X$ via raising to the $p$-th power as $p \neq 2$.
Then $(\mathcal{N}, \pi^*\alpha)$ is a Tango structure of index $2$ on $C$ where the theta characteristic $\mathcal{N}^{\otimes p}$ is globally generated.
\end{proof}

We finally apply Theorem \ref{thm: criterio_Kodaira_adelic_rational} to Raynaud surfaces.

\begin{theorem} \label{main thm raynaud end paper}
    Let $K$ be a global field of characteristic $p>0$. 
    Let $X$ be the Raynaud surface associated to a Tango structure $(\mathcal{L}, \mathcal{N}, \alpha)$ of index $2$ on a curve $C$  and theta characteristic $\mathcal{N}^{\otimes p}$ globally generated.
    Suppose that the curve together with the Tango structure descends to $K^p$, giving a model $(\widetilde{X}, \widetilde{\mathcal{F}})$ of $(X, \mathcal{F})$.
    If $\phi \colon Y \coloneqq \widetilde{X}/\widetilde{\mathcal{F}} \to X$ denotes the natural morphism, then we have
    $$(X(\mathbf{A}_K) \setminus \phi(Y(\mathbf{A}_K)))^{\Br(X)[p]} =  X(K) \setminus \phi(Y(K)).$$
\end{theorem}

\begin{proof}
    By Proposition \ref{prop: bpf_lin_system}, the line bundle $\mathcal{A}^{\otimes p}$ is globally generated and $\psi_{\mathcal{A}^{\otimes p}}$ separates points.
    The result  follows from Theorem \ref{thm: criterio_Kodaira_adelic_rational} to conclude. 
\end{proof}

\bibliographystyle{amsalpha}
\bibliography{refs}

\end{document}